\documentclass[reqno,11 pt]{amsart}
\usepackage{amsmath,amssymb,latexsym}
\usepackage[colorlinks=true, linkcolor={black}]{hyperref}
\usepackage{amsthm}
\usepackage{comment}
\usepackage{color}  
\usepackage[all]{xy}
\usepackage{tikz-cd}
\usepackage{multirow}
\usepackage{multicol}
\usepackage{longtable}
\usepackage{caption}
\usetikzlibrary{cd}
\usepackage{float}
\usepackage[utf8]{inputenc}

\hypersetup{
    colorlinks=true, 
    linktoc=all,     
    linkcolor=blue,  
    urlcolor=black,
}

\newcommand{\F}{{\mathbb{F}}}
\newcommand{\Z}{{\mathbb{Z}}}
\newcommand{\Q}{{\mathbb{Q}}} 

\newcommand{\N}{{\mathbb{N}}}    
    
\newcommand{\C}{{\mathbb{C}}}    

\newcommand{\q}{{\mathfrak{q}}}
\newcommand{\p}{{\mathfrak{p}}}

\newcommand{\OO}{{\mathcal{O}}}

\newcommand{\lra}{\longrightarrow}

\newcommand{\Hilbert}[3]{\left( #1, #2 \right)_{#3}}

 \newcommand{\genlegendre}[4]{%
  \genfrac{(}{)}{}{#1}{#3}{#4}%
  \if\relax\detokenize{#2}\relax\else_{\!#2}\fi
}
\newcommand{\legendre}[3][]{\genlegendre{}{#1}{#2}{#3}}

\makeatletter
\newcommand*{\rom}[1]{\expandafter\@slowromancap\romannumeral #1@}
\makeatother

\DeclareFontFamily{U}{wncy}{}
    \DeclareFontShape{U}{wncy}{m}{n}{<->wncyr10}{}
    \DeclareSymbolFont{mcy}{U}{wncy}{m}{n}
    \DeclareMathSymbol{\Sh}{\mathord}{mcy}{"58}
		
\theoremstyle{plain}

\newtheorem{theorem}{Theorem}[section]
\newtheorem*{theorem*}{Theorem}
\numberwithin{equation}{section}
\newtheorem*{theorem'''}{Proposition {\bf C}}
\newtheorem*{theorem'}{Theorem {\bf A}}
\newtheorem*{theorem''}{Corollary {\bf B}}
\newtheorem*{theorem''''}{Theorem {\bf D}}

\newtheorem{proposition}[theorem]{Proposition}
\newtheorem{rem}[theorem]{Remark}
\newtheorem{example}[theorem]{Example}
\newtheorem{lemma}[theorem]{Lemma}

\newtheorem{defn}[theorem]{Definition}

\begin{document}

\title[cube Sum and class Number of cubic fields]{On certain root number  $1$  cases of the cube sum problem}

\author[Shamik Das]{Shamik Das}
\email{shamikd@iitk.ac.in, jhasom@iitk.ac.in}

\author[Somnath Jha]{Somnath Jha}
\address{Department of Mathematics and Statistics\\ IIT Kanpur \\ India}


\date{}

\begin{abstract}
We consider certain  families of integers $n$ determined by some congruence condition, such that the global root number of the elliptic curve $E_{-432n^2}: Y^2=X^3-432n^2$ is $1$ for every $n$, however a given  $n$ may or may not be a sum of two rational cubes.  
We give explicit criteria in terms of the $2$-parts  and $3$-parts of the ideal class groups of certain cubic number fields to determine whether such an $n$ is a cube sum.  In particular, we study integers $n$ divisible by $3$ such that the global root number of $E_{-432n^2}$ is $1$. For example, for a prime $\ell \equiv 7 \pmod{9}$, we show that for $3\ell$ to be a sum of two rational cubes, it is necessary that the ideal class group of $\Q(\sqrt[3]{12\ell})$ contains $\frac{\Z}{6\Z}\oplus \frac{\Z}{3\Z}$ as a subgroup. Moreover, for a positive proportion of primes $\ell \equiv 7 \pmod{9}$, $3\ell$ can not be a sum of two rational cubes. A key ingredient in the proof is to explore the relation between the $2$-Selmer group and the $3$-isogeny Selmer group of $E_{-432n^2}$ with the ideal class groups of appropriate cubic number fields. 
\end{abstract}

\subjclass[2020]{}

\maketitle

\section*{Introduction}
\noindent  
 An integer $n$ is said to be a rational cube sum or simply a cube sum if $n=x^3+y^3$ for some $ x,y \in \Q$. If an integer $n$ can not be written as a sum of two rational cubes, then we say that $n$ is a non-cube sum. A classical Diophantine problem asks the question: which integers are cube sums? It is  well-known that a cube-free integer $ n > 2 $ is a cube sum if and only if the elliptic curve $$ E_{-432n^2}:  y^2 = x^3 -432n^2 $$ has  positive Mordell-Weil rank over $ \mathbb{Q} $ i.e.  $\operatorname{rank}_\Z \ E_{-432n^2}(\Q) >0 .$ A recent important result of Alp\"oge-Bhargava-Shnidman-Burungale-Skinner  \cite{abs} shows that a positive proportion of integers are cube sums and a positive proportion of integers are not. 
 Let $ w(n) =w (E_{-432n^2}/\Q)  \in \{ \pm 1 \} $ denote the global root number of the elliptic curve $ E_{-432n^2} $ over $ \mathbb{Q} $ i.e. $w(n)= (-1)^{\operatorname{ord}_{s=1} \ L(E_{-432n^2}/\Q, s)}$, the sign of the functional equation of the Hasse-Weil complex $L$-function 
$L(E_{-432n^2}/\Q,s)$ of $E_{-432n^2}$ over $\Q$ (see \cite{roh}). 
For a cube-free integer $n >2$, a computation by Birch-Stephens \cite{bs} gives an explicit formula for $ w(n) $, as follows:  
\begin{small}
\begin{equation}
 \label{rootnumber}  
 w(n) = -\prod\limits_{p \text{ prime}} w_p(n), \quad \text{ where}
\end{equation}
$$
w_3(n) =
\begin{cases} 
-1, & \text{if } n \equiv \pm 1, \pm 3 \pmod{9}, \\ 
1, & \text{otherwise,} 
\end{cases}  
\quad \text{ and  for $p \neq 3$,} \quad 
w_p(n) =
\begin{cases} 
-1, & \text{if } p \mid n \text{ and } p \equiv  2 \pmod{3}, \\ 
1, & \text{otherwise.} 
\end{cases}
$$
\end{small}
 Let us denote the algebraic and analytic rank of $E_{-432n^2}/\Q$ by  $r_{\operatorname{al}}(n)$ and $r_{\operatorname{an}}(n)$, respectively i.e. $r_{\operatorname{al}}(n):=\text{rank}_\Z \ E_{-432n^2}(\Q)$  and $r_{\operatorname{an}}(n):= \text{ord}_{s=1} \ L(E_{-432n^2}/\Q,s)$.  Then (a part of) the Birch and Swinnerton-Dyer (BSD) conjecture predicts that $r_{\operatorname{al}}(n) =   r_{\operatorname{an}}(n)$ and the parity conjecture asserts that $r_{\operatorname{al}}(n) \equiv  r_{\operatorname{an}}(n) \pmod 2.$   
Thus, if the root number $w (n)=-1 $, the parity conjecture predicts that $r_{\operatorname{al}}(n) >0$. However, if  $w (n)=1 $, then the situation is ambiguous and $r_{\operatorname{al}}(n)$ may either be $0$ or a positive even integer.

Albeit the important result in \cite{abs}, there is no unconditional general method or a provably terminating algorithm to determine if a given general integer $n$ is a cube sum \footnote{Although, in practice, given an elliptic curve $E/\Q$, assuming the finiteness of the Tate-Shafarevich group $\Sh(E/\Q)$ (Definition \ref{eq:defofsha}), one can compute the rank of  $E(\Q)$  using a standard computer software if the conductor of $E$ is reasonably small.}. There are classical works of Sylvester \cite{syl} and Selmer \cite{sel} on this  topic and most of the available literature covers the case where the (cube-free) integer $n$ is of particularly `simple' form, given by $p^iq^j$ with $p,q$  distinct primes and $i,j \leq 2$ (cf. \cite{dv}, \cite{jms2}). Further, in this case, if the root number $ w(n)=-1$, then the Heegner point type of argument has been used in the literature, following \cite{sat} (cf. \cite{dv}, \cite{yi}). In this article, we focus on  `ambiguous' cases of the cube sum problem for certain  families of cube-free integers $n$ such that the global root number $w(n)$ of the elliptic curve $E_{-432n^2}/\Q$  is  $1$ for every $n$ in a family, but a given $n$ in the family may or may not be a cube sum. Further, an $n$ as above divisible by $3$, is of particular interest to us.

At first, let us consider the prime numbers. Then the cube sum  property of primes is studied in congruence classes modulo $9$ and is governed by the so called Sylvester's conjecture \cite{syl} (cf. \cite{dv}).  For a prime number $\ell$, it follows from \eqref{rootnumber} that  the root number $w(\ell)=1$ if  $\ell  \equiv 1, 2 \text{ or } \ 5 \pmod 9$. Further, if $\ell  \equiv  2 ,  5 \pmod 9$, it was shown by P\'epin, Lucas and Sylvester \cite{syl} that $\ell$ can not be written as a sum of two rational cubes. However, the situation for a prime $\ell \equiv 1 \pmod 9$ is ambiguous and $r_{\operatorname{al}}(\ell)$ may be $0$ or a positive even integer (and there are examples of both).  Villegas-Zagier \cite{rz} studied the case of a prime $\ell \equiv 1 \pmod 9$  and presented three different efficient methods to determine whether $L(E_{-432\ell^2}/\Q,s)$ vanishes at $s=1$ or not. Note that if $L(E_{-432\ell^2}/\Q,s)$ vanishes at $s=1$, to conclude $r_{\operatorname{al}}(\ell)>0$, one needs to invoke the BSD conjecture, which is wide open for $r_{\operatorname{an}}(\ell) \geq 2$. Note that using binary cubic forms, it was shown in \cite{jms2} that there are infinitely many primes $\ell \equiv 1 \pmod 9$ such that $\ell$ is a cube sum i.e. $r_{\operatorname{al}}(\ell) >0$, although the set of such primes is not explicit there.

More generally, when we have an infinite family $\mathcal F$ of (cube-free) integers, such that $ E_{-432n^2}$ has global root number equal to $1$ for every $n \in \mathcal F$ and $ \mathcal F$ contains both cube sum and non-cube sum integers, it would be useful to have explicit criteria for verifying whether a given $n \in \mathcal F$ is a cube sum or not. Note that the elliptic curve $E_{-432n^2}$ has additive reduction at the prime $3$  for any $n$ and further, if (i) $3\mid n$, (ii) the root number $w(n)$ is $ 1$ for every $n \in \mathcal F$ and (iii) $\mathcal F$ contains integers $n$ such that rank of $E_{-432n^2}(\Q)$ is positive (respectively zero), then cube sum problem for such a family $\mathcal F$ does not seem to be discussed in the literature. 

 In the main results of this article (Theorem \hyperlink{thm:A}{A} and Corollary \hyperlink{thm:B}{B}), we discuss a necessary condition for an integer of the form $3\ell$ or $3\ell^2$, where $\ell $ is a prime varying in certain congruence class modulo $  9$ to be a cube sum, in terms of the $2$-part  and $3$-part of the ideal class group of a certain cubic number field.  As a by-product, the criterion gives us an estimate of the density of non-cube sum integers in the family. We fix some notation before stating the result.

 {\bf Notation:} We say that a positive integer $n$ is cube-free if $p^3 \nmid n $ for any prime $p$. Throughout the article,  $\operatorname{cf}(n)$ will denote the cube-free part of a positive integer $n$ i.e. $\operatorname{cf}(n)=\frac{n}{m^3}$,  where $m$ is the largest positive integer such that $m^3\mid n$. For a cube-free integer $n  > 1$, let $\operatorname{Cl}_{\mathbb{Q}(\sqrt[3]{n})}$ be the ideal class group of the cubic number field $\mathbb{Q}(\sqrt[3]{n})$. Let $A$ be an abelian group and $p$ be a prime. For any $n \in \N$, recall $A[p^n]:=\{x\in A: p^nx=0\}$ and the $p$-rank of $A$ is defined to be $\text{dim}_{\mathbb F_p}\ A\otimes_\Z {\mathbb F_p}=\text{dim}_{\mathbb F_p}\ A[p].$ 
\begin{defn}
\label{p-rank-cl}
    For a prime number $p$, we denote by $h_p(n)$, the $p$-rank of $\operatorname{Cl}_{\mathbb{Q}(\sqrt[3]{n})}$.   
\end{defn}

\begin{theorem'}
\hypertarget{thm:A}{}
\label{thm3l=7,4}
Let $\ell$ be a prime.
\item[(i)] If $\ell \equiv 7 \pmod{9}$ and $3\ell$ is a cube sum, then $h_3(12\ell)=2$. Moreover, for a positive proportion of primes  $\ell  \equiv 7 \pmod{9}$,  $3\ell$ is not a cube sum.
\item[(ii)] If $\ell \equiv 4 \pmod{9}$ and $3\ell^{2}$ is a cube sum, then $h_3(18\ell) =2$. Moreover, for a positive proportion of primes  $\ell  \equiv 4 \pmod{9}$,  $3\ell^2$ is not a cube sum. 
\end{theorem'}
Strengthening Theorem \hyperlink{thm:A}{A}, we have the following corollary:
\begin{theorem''}
\hypertarget{thm:B}{}
\label{mod18}
Let $\ell$ be a prime.
\item[(i)] If $\ell \equiv 7 \pmod{9}$ and $3\ell$ is a cube sum, then $\operatorname{Cl}_{\Q(\sqrt[3]{12\ell})}$ contains a subgroup isomorphic to $\frac{\Z}{6\Z}\oplus \frac{\Z}{3\Z}$. 
\item[(ii)] If $\ell \equiv 4 \pmod{9}$ and $3\ell^2$ is a cube sum, then \begin{small}$\operatorname{Cl}_{\Q(\sqrt[3]{18\ell})}$\end{small} contains a subgroup isomorphic to $\frac{\Z}{6\Z}\oplus \frac{\Z}{3\Z}.$ 
\end{theorem''}

 In fact, Corollary \hyperlink{thm:B}{B} follows from Theorem  \hyperlink{thm:A}{A}  and Proposition \hyperlink{thm:C}{C}, which we state below. We prove a more general result in Proposition \hyperlink{thm:C}{C} and in particular, the proposition  yields a necessary condition for a prime $\ell \equiv 1 \pmod 9$ to be a cube sum in terms of $h_2(4\ell)$.

\begin{theorem'''}
\hypertarget{thm:C}{}
\label{cubesum2part}
Let $n >2 $ be a cube-free  integer which is a rational cube sum. Assume the following:  
(i) {$\operatorname{cf}(4n) \not\equiv 
\pm 1 \pmod{9}$,} and  
(ii) the global root number $w(n)$ of $E_{-423n^2}$ over $\Q$ is equal to $1$. 
Then $h_2(\operatorname{cf}(4n)) >0$ i.e. the class number of  $\mathbb{Q}(\sqrt[3]{4n})$ is even. 

In particular, let $\ell \equiv 1 \pmod 9$ be a prime. If $h_2(4\ell)=0 $ (respectively $h_2(2\ell)=0$), then $\ell$ (respectively $\ell^2$) is a non-cube sum.\qed
\end{theorem'''}

We also establish a result similar to Theorem \hyperlink{thm:A}{A} for integers of the form $2\ell$ and $2\ell^2$, where $\ell$ is a prime in certain congruence class modulo $9$:

\begin{theorem''''}
\hypertarget{thm:D}{}
\label{thm2l=1}
Let $\ell$ be a prime satisfying $\ell \equiv 1 \pmod{9}$. If either $2\ell$ or $2\ell^2$ is a cube sum, then $h_3(2\ell) = 2$. Furthermore, for a positive proportion of primes $\ell \equiv 1 \pmod{9}$, neither $2\ell$ nor $2\ell^2$ can be expressed as a sum of two rational cubes.
\end{theorem''''}

Let $E$ be an elliptic curve over a number field $K$.  The Mordell-Weil group of $E(K)$ is difficult to compute and given a $K$-rational  isogeny $\varphi: E\lra \hat{E}$, via the $\varphi$-descent exact sequence (see \eqref{eq:defofsha}), often one instead studies the  $\varphi$-Selmer group $S_\varphi(E/K)$ (Definition \ref{mainsel}).   
Starting with the work of Cassels \cite{cass}, the relation between an isogeny induced Selmer group of $E/K$ and the ideal class group of a suitable extension of $K$ has been studied extensively by various authors (see \cite{bk}, \cite{ss} and also \cite{jms}). In our case for $E_{-432n^2}/\Q$, we have a rational degree $3$-isogeny  $\varphi_n: E_{-432n^2} \lra E_{16n^2}$ (see \cite{bes}, \cite[\S2]{jms}, also \eqref{eq:defofphi}). 
The broad idea behind the proofs of our main results 
is to explore the relation between the $2$-Selmer group (respectively, the $\varphi_n$-Selmer group) of $E_{-432n^2}$ with the $2$-part  (respectively, $3$-part)  of the  ideal class group of appropriate cubic number fields. However, we would like to mention the following:
    \begin{rem}{\rm To compare the $2$-Selmer group of $E_{-432n^2}/\Q$ with the ideal class group of a suitable number field $F$, it is a natural choice to consider  \begin{small}$F:=\frac{\Q[X]}{(X^3-432n^2)}\cong \Q(\sqrt[3]{4n})$\end{small}, as done in Proposition \hyperlink{thm:C}{C}. On the other hand, $E_{-432n^2}[3]$ is a reducible $G_\Q$-module and  a degree-$3$ isogeny corresponds to a $G_\Q$ stable subgroup of order $3$ in $E(\bar{\Q})$. It is easy to verify that the $3$-torsion points of $E_{-432n^2}$ are defined over $\Q(\sqrt{-3}, \sqrt[3]{n})$. However, the cubic fields stated in Theorem \hyperlink{thm:A}{A}  and in Theorem \hyperlink{thm:D}{D}  for the case $2\ell^2$ are not contained in $\Q(\sqrt{-3}, \sqrt[3]{n})$. 
    
    Let us discuss a couple of examples;  $3\cdot 61$ is a cube sum with $61\equiv 7 \pmod 9$. However, the class numbers of both $\mathbb{Q}(\sqrt[3]{183})$ and $\mathbb{Q}(\sqrt{-3}, \sqrt[3]{183})$ are equal to $3$. On the other hand,  $3\cdot43$ is a non-cube sum with $3$-ranks of the class groups of both $\mathbb{Q}(\sqrt[3]{129})$ and $\mathbb{Q}(\sqrt{-3}, \sqrt[3]{129})$ are equal to $1$; so the cube sum property is not captured via the $3$-part of the class groups of subfields inside $\Q(\sqrt{-3}, \sqrt[3]{n})$. 
    
    Thus, it requires some work to the make the correct choice of the fields, which yield Theorem  \hyperlink{thm:A}{A}   and Theorem \hyperlink{thm:D}{D}.}
\end{rem}

\begin{rem}\label{converse}
{\rm 
\begin{itemize} 
\item We emphasize that in each of the families considered in the above results, the root number $w(n)$ of the corresponding elliptic curve is always equal to $1$, but there are examples of both cube sum and non-cube sum integers (see table \ref{tab:class_numbers}).
\item In Example \ref{ex1}, we illustrate that the necessary condition obtained in our results above are not sufficient. 
\item We demonstrate, via Example \ref{ex2}, that both the assumptions (i) and (ii) are necessary in  Proposition \hyperlink{thm:C}{C}.
\item Note that we have relaxed the condition (i) of Proposition \hyperlink{thm:C}{C} in the setting of Theorem \hyperlink{thm:D}{D}.
\end{itemize}}
\end{rem}
We now discuss the idea behind the proofs of the results stated above; starting with Theorem \hyperlink{thm:A}{A}.  The proof of the first assertion of Theorem \hyperlink{thm:A}{A} is divided in two steps. At first, we show that the structure of the $3$-part of the ideal class group of the cubic fields stated in our results, is related to the cubic residue symbol of 
$3$ modulo the corresponding prime (Lemma \ref{lem3l=7,4,3rank}). This step uses results of Gerth \cite{ger} along with some explicit computation on relevant cubic Hilbert symbols. 
 
 As stated earlier, the Mordell curve $E_{-432n^2}$ has a  $3$-isogeny $\varphi_n: E_{-432n^2} \lra E_{16n^2}$ (see \eqref{eq:defofphi}). The idea in the second step is to explicitly compute this $3$-isogeny Selmer group $S_{\varphi_n}(E_{-432n^2}/\Q(\sqrt{-3}))$ of $E_{-432n^2}$ over $\Q(\sqrt{-3})$,  with $n$ in the setting of Theorem \hyperlink{thm:A}{A} (a suitable description of the Selmer group in this setting is given in \eqref{eq:newseldefsq}). In fact, in Proposition 
 \ref{prop3l=7,4,cubic},  we relate the $\F_3$-dimension of this Selmer group with  the same cubic residue symbol appearing in the first step of the proof; thereby completing  the argument (for the first assertion of Theorem \hyperlink{thm:A}{A}).

The elliptic curve $E_{-432n^2}/\Q$ in general has bad, additive reduction at $3$. Further, in the setting of  Proposition \ref{prop3l=7,4,cubic}, $3\mid n$ and the image of the Kummer map of $E_{-432n^2}$ at the prime above $3$ is difficult to determine precisely (see Remark \ref{lastrem},  \cite[Prop. 4.10(2)]{jms}, \cite{dmm}).  So, we only get an upper bound on the $\F_3$-dimension of $S_{\varphi_n}(E_{-432n^2}/\Q(\sqrt{-3}))$ in Proposition \ref{prop3l=7,4,cubic} and then appeal to the $3$-parity conjecture (known due to  Nekov\'a\v{r}, Kim, Dokchitser-Dokchitser, cf. \cite{nek})  to compute the dimension precisely. In fact, as the root number $w(n)=1$,   the $3$-parity conjecture (Theorem \ref{s2w}) gives us that $\operatorname{dim}_{\Q_3}\ \operatorname{Hom}_{\Z_3}\big(S_{3^\infty}(E_{-432n^2}/\Q),\Q_3/{\Z_3}\big)\otimes_{\Z_3} {\Q_3}$ is even. Here for a prime $p$, $S_{p^\infty}(E_{-432n^2}/\Q)$ denotes the  $p^\infty$-Selmer group of $E_{-432n^2}/\Q$, defined in  \eqref{p-primary-sel}. Then we compare the parity of the corresponding ranks of the $3^\infty$ and $3$-Selmer groups of $E_{-432n^2}/\Q$ in Lemma \ref{2and2infinity}, using the Cassels-Tate pairing  on the Tate-Shafarevich group $\frac{\Sh(E_{-432n^2}/\Q)}{\Sh(E_{-432n^2}/\Q)_\text{div}}$  (see \ref{n-selmer-d}).  Note that using the arithmetic of the  elliptic curve $E_{-432n^2}$ and  the above Cassels-Tate pairing on $\Sh$, we can relate the $\F_3$-dimensions of  $\ S_{3}(E_{-432n^2}/\Q)$ and   $S_{\varphi_n}(E_{-432n^2}/\Q(\sqrt{-3})).$

The methods to prove the first part of Theorem \hyperlink{thm:D}{D} is similar in spirit to that corresponding part of  Theorem \hyperlink{thm:A}{A}. At first, in this setting, $h_3(n)$ is related to the cubic residue symbol of $2$ in Lemma \ref{lem2l=1,3rank}. However, as $3\nmid n$, the image of the Kummer map at $3$ for $E_{-432n^2}$ can be determined and under some suitable assumption (which appears in Lemma \ref{lem2l=1,3rank}),  $\text{dim}_{\F_3} \  S_{\varphi_n}(E_{-432n^2}/\Q(\sqrt{-3}))$ has been computed precisely in \cite[Theorem 1.2]{jms} and we use this result to deduce Theorem \hyperlink{D}{D}. 

 For the second assertions relating to the positive density of primes in Theorems \hyperlink{thm:A}{A} and \hyperlink{thm:D}{D}, we recall that there are classical results which relate the cubic residue symbol of $2$ (respectively $3$) modulo a prime $\ell \equiv 1 \pmod 3$ with the representation of the prime $\ell$ by certain integral binary quadratic form (cf. \cite{cox}). It is known that the subset of primes congruent to $1 \pmod 3$, represented by these integral binary quadratic form, has a positive (Dirichlet) density. However, we need a refinement of this statement to complete our proof. Specifically,  we need to show these integral binary quadratic forms represent a subset of primes of positive (Dirichlet) density, in an arithmetic progression determined by a given congruence class modulo $9$. We extract this result from the work of \cite{hal}, which is an  extension of the results of \cite{mey}.

We now outline the proof of Proposition \hyperlink{thm:C}{C}. Let $E/\Q$ be an elliptic curve and let $\Q(E[2])$ be the field obtained by adjoining the $2$-torsion points of $E$ over $\Q$. Assume that $E[2](\Q)=0$. Then a result of \cite[Proposition 7.1]{bk} (see \eqref{2rankbound}) relates the $\F_2$-dimension of the $2$-Selmer group $S_2(E/\Q)$ with the $2$-part of the ideal class group of a  cubic subfield, say $F$, of $\Q(E[2])$. Assuming $n$ to be a cube sum in Proposition \hyperlink{thm:C}{C}, we get that $\text{rank}_\Z \ E_{-432n^2}(\Q) >0$. Then applying the $2$-parity conjecture (known due to Kramer, Monsky,  cf. \cite{mon}), we determine the parity of $\operatorname{dim}_{\Q_2}\ \operatorname{Hom}_{\Z_2}\big(S_{2^\infty}(E_{-432n^2}/\Q),\Q_2/{\Z_2}\big)\otimes_{\Z_2} {\Q_2}$ and further, using Lemma \ref{2and2infinity}, we compare it with the parity of $\text{dim}_{\F_2} \ S_{2}(E_{-432n^2}/\Q)$. Analyzing the reduction types of the CM elliptic curve $E_{-432n^2}$, identifying $\mathrm{dim}_{\F_2}\ \mathrm{Cl}_F[2]$ with $h_2(\mathrm{cf}(4n))$ and applying \cite{bk}'s result, we deduce the proposition.

 \noindent \textbf{Structure of this article:} After the introduction,  \S\ref{prelim} contains preliminaries on  (i) the relation between $3$-rank of the ideal class group of a cubic field and cubic Hilbert symbols,  and (ii)  Selmer groups and the $p$-parity conjectures.  In  \S \ref{proof1.2}, we prove all our results, stated in the introduction. A table  of numerical examples (Table \ref{tab:class_numbers}) appears in \S \ref{nexampl}.
 
 \noindent \textbf{Acknowledgment:} It is a pleasure to thank Dipramit Majumdar for discussions and answering many questions.  We are also grateful to Debanjana Kundu and Pratiksha Shingavekar for comments and suggestions. We thank the anonymous referee for various comments, suggestions and corrections which has helped us to improve the article. S. Das is supported by DST INSPIRE Faculty fellowship. S. Jha 
acknowledges the support of ANRF grant CRG/2022/005923. 
\section{preliminaries}\label{prelim}
In this section, we recall some definitions, discuss the basic set up and also state some known results which are used later.
\subsection{\texorpdfstring{$3$}{lg}-part of the class number of \texorpdfstring{$\Q(\sqrt[3]{n})$}{lg}}\label{cubic-fields} 
For any number field $M$, let $\mathcal{O}_{M}$ be its ring of integers. Let $\zeta=\zeta_3$ be a primitive cube root of unity in $\mathbb{C}$ and put $K:=\mathbb{Q}(\zeta)=\Q(\sqrt{-3})$. Set $\mathfrak{p}: =1-\zeta$ and by a slight abuse of notation, we denote both the element $\p$ and the ideal $(\p)$ by $\p$ and it is understood from the context.   Observe that $3\mathcal{O}_{K} = \mathfrak{p}^{2}$. Let $n >1 $ be a  cube-free  integer. Put  $F:=\mathbb{Q}(\sqrt[3]{n})$ and $L := K(\sqrt[3]{n}) = \mathbb{Q}( \zeta, \sqrt[3]{n})$. 
At first we note down well-known results on the ramification of rational primes in certain  cubic number fields which can be conveniently found in \cite[Tables 1-3]{sch}.  
\begin{lemma} 
\label{ramification} 
 Let $n > 1$ be a cube-free integer and $F$  and $L$ be as above.  Then we have:
\begin{itemize} 
\item[(i)] Let $q \neq 3$ be a prime in $\mathbb{Z}$. If $q \mid n$, then $q\mathcal{O}_F = \mathfrak{Q}^3$, where $\mathfrak{Q}$ is a prime of $\OO_F$ above $q$. 
\item[(ii)]  If $n^2 \not\equiv 1 \pmod{9}$, then $3$ is totally ramified in both $F$ and $L$. In particular, this holds if $3\mid n$.
\item[(iii)] If $n^2 \equiv 1 \pmod{9} $, then $3\mathcal{O}_F = \mathfrak{PQ}^2$, where $\mathfrak{P}$ and $\mathfrak{Q}$ are distinct primes of $\mathcal{O}_F$. Also, in this case $\mathfrak{p}\mathcal{O}_{L}= \mathfrak{P}_{1}\mathfrak{P}_{2}\mathfrak{P}_{3}$, where $\mathfrak{P}_{1}, \mathfrak{P}_{2}$ and $\mathfrak{P}_{3}$ are distinct primes in $\mathcal{O}_{L}.$ 
\end{itemize} 
\end{lemma}

 Now, following  \cite{ger}, we give an explicit formula for $h_3(n)$, the $3$-rank of $\operatorname{Cl}_{\Q(\sqrt[3]{n})}$ using cubic Hilbert symbols. First, we introduce some notation and discuss the set up.
Consider a positive integer $n$ in the following form:
\begin{equation}
\label{n}    
n=2^{f}3^{\mu}p^{e_{1}}_{1}\cdots p^{e_{v}}_{v}p^{e_{v+1}}_{v+1}\cdots p^{e_{w}}_{w},\end{equation}
where the $p_{i}$ and $q_{i}$ are (positive) integer  primes,
\begin{small}
\begin{align*}
 p_{i} \equiv 1  \pmod{9} \quad \text{ for } 1 \leq i \leq v, \quad \text{ and } \quad
 p_{i} \equiv 4,\;7   \pmod{9} \quad \text{ for } v+1 \leq i \leq w,
\end{align*}
\end{small}
with $e_{i}, f \in  \{1,2\}$ and $\mu \in \{0,1,2\}$. Recall that for $1\leq i\leq w$, each $p_{i}$ splits as  $p_{i}=\pi_{i} \pi_{i}'$ in $\mathcal{O}_K=\Z[\zeta]$, where $\pi_{i}$ and $\pi_{i}'$ are prime elements, each congruent to $1 \pmod {3\mathcal{O}_K}$, and are complex conjugates of each other. Observe that in Theorems \hyperlink{thm:A}{A} and \hyperlink{thm:D}{D}, we consider cubic fields of the form $\mathbb{Q}(\sqrt[3]{n})$, where $n$ can be expressed in the form given by \eqref{n} and $n^2 \not \equiv 1\pmod{9}$.

\noindent  Following \cite{ger}, for an integer $n$ of the form given in \eqref{n}, we define $2w$-tuples $x(n)=(x_{1}, x_{2}, \ldots, x_{2w})$, $x_i\in K$ as follows: 
\begin{align}
\label{x}
(x_{1}, x_{2}, \ldots, x_{2w})= 
\begin{cases}
   &( \pi_{1}\pi_{1}^{\prime 2}, \ldots, \pi_{w}\pi_{w}^{\prime 2}, p_{1}, \ldots, p_{v}, \; p_{v+1}p^{h_{v+2}}_{v+2}, \ldots, p_{v+1}p^{h_{w}}_{w}, p_{v+1} 2^{\alpha}) \; \text{ if } w>v,\\
 &( \pi_{1}\pi_{1}^{\prime 2}, \ldots, \pi_{w}\pi_{w}^{\prime 2}, p_{1}, \ldots, p_{w}) \; \text{ if } w=v.  
\end{cases}
\end{align}
Here, for $w>v$ and for each $i$ with $v+2 \leq i \leq w $,  $h_{i}$ is defined as follows: $h_{i} \in \{1,2\}$ and $h_{i}$ is chosen so that $p_{v+1}p^{h_{i}}_{i}\equiv 1 \pmod{9}$ holds.  Also, $\alpha \in \{1, 2\}$ is chosen so that $2^{\alpha}p_{v+1}\equiv \pm 1 \pmod{9}$.

 Let $n$ be an integer of the form given in \eqref{n}. Let $x = (x_1, \ldots, x_{2w})$ be as given in \eqref{x}, and set $u := 2w + 2$. We define a $w \times u$ matrix $B = (\beta_{ij})$ over the field $\mathbb{F}_3$ as follows:
\begin{equation}
\label{B}
 \zeta^{\beta_{ij}}=\begin{cases}
 & \Hilbert{x_{w+i}}{n}{\pi_{m}} \quad 1 \leq i \leq w, \quad  1\leq m \leq w, \;\; j=2m-1,\\
 & \Hilbert{x_{w+i}}{n}{\pi'_{m}} \quad 1 \leq i \leq w, \quad  1\leq m \leq w, \;\; j=2m,\\
 & \Hilbert{x_{w+i}}{n}{2} \quad \quad 1 \leq i \leq w, \;\; j=2w+1,\\
 & \Hilbert{x_{w+i}}{\mathfrak{p}}{\mathfrak{p}}  \quad \quad 1 \leq i \leq w,\;\: j=2w+2 \quad \text{if  $n^2 \not\equiv 1 \pmod{9}$}. 
 \end{cases}
\end{equation}
The symbol $\Hilbert{a}{b}{\mathfrak{\pi}}$ is the cubic Hilbert symbol, where $a,b \in K^{*}$ and $\pi$ is a prime of $\OO_K$. Let $v_{\pi}$ denote the $\pi$-adic valuation. Then the Hilbert symbol is computed as follows:
\begin{align}
\label{Hilbertsymbol}
\Hilbert{a}{b}{\pi}=\legendre[3]{c}{\pi}, \text{ where } c=(-1)^{v_{\pi}(a)v_{\pi}(b)}a^{v_{\pi}(b)}b^{-v_{\pi}(a)}
\end{align}
and $\legendre[3]{*}{*}$ is the cubic reciprocity symbol (see \cite{lem} for details). By Lemma \ref{ramification}, $\mathfrak{p}$ ramifies in $L = \mathbb{Q}( \zeta, \sqrt[3]{n})$  if and only if $n^2 \not\equiv 1 \pmod{9}$ and using this, it follows that the definition of $B$ matrix in \eqref{B} is consistent with \cite[\S4]{ger}.  With this set up, we are now ready to express $h_3(n)$ in terms of the cubic Hilbert symbols.
\begin{lemma} \cite[Lemma 4.4]{ger}
\label{cubic3rank} Let $n$ be an integer of the form given in \eqref{n}. Then the $3$-rank of  $\operatorname{Cl}_{\Q(\sqrt[3]{n})}$  is given by $ h_{3}(n)=2w -\text{rank } B$, where $ B $ is the $w\times u$ matrix over $\F_3$, defined in \eqref{B}.
\end{lemma}

\subsection{Isogeny induced Selmer groups and the \texorpdfstring{$p$}{lg}-parity conjecture}\label{subsecselmer}
 Throughout, we fix an embedding  $\iota_\infty: \bar{\mathbb{Q}} \hookrightarrow \mathbb{C}$ of a fixed algebraic closure  $\bar{\mathbb{Q}}$ of $\mathbb{Q}$ into $\C$ and also an embedding $\iota_p: \bar{\mathbb{Q}} \hookrightarrow \bar{\mathbb{Q}}_p$; into a fixed algebraic closure $\bar{\mathbb{Q}}_p$ of $\mathbb{Q}_p$ for every prime number $p$. Let $F$ be a number field, and let $\Omega_F$ denote the set of all (Archimedean and non-Archimedean) places of $F$. For each place $v \in \Omega_F$, let $F_v$ denote the completion of $F$ at $v$. For $T \in \{F, F_v\}$, denote by $G_T := \operatorname{Gal}(\bar{T}/T)$ the absolute Galois group of $T$.

 Let $E$, $\widehat{E}$ be elliptic curves over $F$ and $\varphi: E \to \widehat{E}$ be an isogeny defined  over $F$. For $T \in \{F, F_{v}\}$, let \begin{small}$\delta_{\varphi, T}:\widehat{E}(T) \longrightarrow \widehat{E}(T)/\varphi(E(T)) \lhook\joinrel\xrightarrow{\bar{\delta}_{\varphi, T}} H^1(G_T, E[\varphi])$\end{small} be the Kummer map. We have the following commutative diagram:
 \begin{small}
\begin{center}\begin{tikzcd}
 0 \arrow[r] & \widehat{E}(F)/\varphi(E(F)) \arrow[r, "\bar{\delta}_{\varphi, F}"] \arrow[d]
& H^1(G_F, E[\varphi]) \arrow[d, "\underset{v\in \Omega_F}{\prod} {\rm res}_v"] \arrow[r] & H^1(G_F,E)[\varphi] \arrow[d] \arrow[r] & 0\\
 0 \arrow[r] & \underset{v \in \Omega_F}{\prod} \widehat{E}(F_v)/\varphi(E(F_v)) \arrow[r, "\underset{v \in \Omega_F}{\prod} \bar{\delta}_{\varphi,F_v}"] & \underset{v \in \Omega_F}{\prod} H^1(G_{F_v},E[\varphi]) \arrow[r] & \underset{v \in \Omega_F}{\prod} H^1(G_{F_v},E)[\varphi] \arrow[r] & 0.
\end{tikzcd}\end{center}
\end{small}
\begin{defn}\label{mainsel}
The $\varphi$-Selmer group of $E$ over $F$, $S_\varphi(E/F)$ is defined as
\begin{small}
$${S}_\varphi(E/F)= \{c \in H^1(G_F,E[\varphi]) \mid {\rm res}_v(c) \in {\rm Image}(\delta_{\varphi,F_v}), \text{ for every } v \in \Omega_F \}.$$\end{small}
\end{defn}
Setting \begin{small}$\Sh(E/F):=\text{Ker}\big( H^1(G_F,E) \to \underset{v \in \Omega_F}{\prod} H^1(G_{F_v},E) \big)$\end{small}, the Tate-Shafarevich group of $E$ over $F$, we get the fundamental exact sequence:
\begin{small}
\begin{equation}\label{eq:defofsha}
0 \longrightarrow {\widehat{E}(F)}/{\varphi(E(F))} \longrightarrow { S}_\varphi(E/F) \longrightarrow \Sh(E/F)[\varphi] \longrightarrow 0
\end{equation}
\end{small}
In particular, for $\varphi=[n]: E(\bar{F}) \overset{\times n}{\longrightarrow} E(\bar{F})$, the multiplication by $n$ map, we have the $n$-Selmer group $S_n(E/F)$. The fundamental $n$-descent exact sequence is given by
\begin{small}
\begin{equation}\label{n-selmer-d}
0 \longrightarrow \frac{{E}(F)}{n(E(F))} \longrightarrow {S}_n(E/F) \longrightarrow \Sh(E/F)[n] \longrightarrow 0.
\end{equation}
\end{small}
We fix a rational prime $p$. Next, we discuss $p^\infty$-Selmer group and the $p$-parity conjecture over $\Q$.  For an abelian group $A$, define $A[p^\infty]:=\underset{n \geq 1}{\cup}A[p^n]$ and set  $E_{p^\infty}:=E(\bar{\Q})[p^\infty]=\underset{n \geq 1}{\cup}E(\bar{\Q})[p^n].$ Then the $p$-primary Selmer group of $E$ over  $\Q$, $S_{p^\infty}(E/\Q)$ is defined by
\begin{small}
\begin{equation}\label{p-primary-sel}
S_{p^\infty}(E/\Q)=\text{Ker}\big(H^1(\Q, E_{p^\infty})\lra \underset{\text{all places } q}{\prod}H^1(\Q_q, E)\big).   \end{equation}
\end{small}
Here, the product is taken over all non-Archimedean and Archimedean  places   of $\Q$. Let $w(E/\Q) \in \{\pm 1\}$ be the global root number of $E$ over $\Q$ (see \cite{roh}). Then the $p$-parity conjecture in this setting states that $\operatorname{dim}_{\Q_p}\ \operatorname{Hom}_{\Z_p}\big(S_{p^\infty}(E/\Q),\Q_p/{\Z_p}\big)\otimes_{\Z_p}\Q_p$ is even if and only if $ w(E/\mathbb{Q})=1 $.  We need the following results establishing the $p$-parity conjecture for an elliptic curve over $\Q$; for $p=2$ it is due to Kramer, Monsky (see \cite[Theorem 1.5]{mon}) and  for an odd prime $p$, due to Nekov\'a\v r, Kim and Dokchitser-Dokchitser (cf. \cite{nek}).  
\begin{theorem}\label{s2w}  
Let $ E/\Q$ be an elliptic curve and $p$ a prime number. Then \begin{small}$$\operatorname{dim}_{\Q_p}\ \operatorname{Hom}_{\Z_p}\big(S_{p^\infty}(E/\Q),\Q_p/{\Z_p}\big)\otimes_{\Z_p} {\Q_p} \text{ is even if  only if }  w(E/\mathbb{Q})=1 .$$\end{small} 
\end{theorem}

\section{Proofs of the main results}
\label{proof1.2}
\noindent In this section, we prove our results stated in the introduction.  We begin with some preparation. At first we discuss the following Lemma:
\begin{lemma}\label{2and2infinity}
Let $ E/\Q$ be an elliptic curve with $E(\Q)[p]=0$ for some prime $p$. Then \begin{small}$$\mathrm{dim}_{\Q_p}\ \mathrm{Hom}_{\Z_p}\big(S_{p^\infty}(E/\Q),\Q_p/{\Z_p}\big)\otimes_{\Z_p} \Q_p \equiv \mathrm{dim}_{\F_p}\ S_p(E/\Q) \pmod 2.$$\end{small}
\end{lemma}
\begin{proof}

 We have  $\Sh(E/\Q)[p^\infty]\cong (\Q_p/{\Z_p})^t\oplus A $, where $t \geq 0 $ and $A$ is a $p$-primary finite abelian  group. Now there is a non-degenerate, alternating (Cassels-Tate) pairing on $p$-primary Tate-Shafarevich group modulo its maximal $p$-divisible subgroup i.e. on $\frac{\Sh(E/\Q)[p^\infty]}{\Sh(E/\Q)[p^\infty]_\text{div}}$ (see \cite[Theorem 1.2]{cass2}) and as a consequence $A\cong B\oplus B$, for some group $B$. It follows that 
 \begin{small}
\begin{equation}\label{sha-p-dim}
     \mathrm{dim}_{\F_p} \ \Sh
(E/\Q)[p] \equiv t \pmod 2.
 \end{equation}\end{small}Recall $r_{\operatorname{al}}(E):=\mathrm{rank}_{\Z}\ E(\Q)$. Using  \eqref{n-selmer-d}, we  get an exact sequence  
 \begin{small}
 \begin{equation}\label{p-infty-sel-exact}
   0\lra   (\Q_p/{\Z_p})^{r_{\operatorname{al}}(E)}\lra S_{p^\infty}(E/\Q)\lra (\Q_p/{\Z_p})^{t}\oplus B \oplus B \lra 0.
 \end{equation}
  \end{small}
On the other hand, it follows from \eqref{n-selmer-d} that 
 \begin{small}\begin{equation}
\text{dim}_{\F_p}\ S_p(E/\Q)= r_{\operatorname{al}}(E) +  \text{dim}_{\F_p}\ E(\Q)[p] + \text{dim}_{\F_p}\ \Sh(E/\Q)[p].
\end{equation}
 \end{small}
Further, using the hypothesis $E(\Q)[p]=0$, we deduce from \eqref{sha-p-dim} that  
 \begin{small}
\begin{equation}\label{p-sel-exact}
 \mathrm{dim}_{\F_p} \ S_p
(E/\Q) \equiv r_{\operatorname{al}}(E)+t \pmod 2.   
\end{equation} \end{small}Now the assertion of the lemma is immediate from \eqref{p-infty-sel-exact} and \eqref{p-sel-exact}.
\end{proof}

 Let $E: y^2=f(x)$ be an elliptic curve over $\Q$ with $E(\Q)[2]=0$. Then the cubic polynomial $f(x)$ is irreducible over $\Q$ and set $F:=\frac{\Q[x]}{(f(x))}$. Then  $F$ is a cubic subfield of $\Q(E[2])$ and any such cubic subfields of $\Q(E[2])$ are  Galois conjugates. Thus $h_2(F)$, the $2$-rank of $\mathrm{Cl}_F$ (Definition \ref{p-rank-cl}), is the same for any cubic subfield $F$ of $\Q(E[2])$.  We recall the following result due to Brumer-Kramer relating $S_2(E/\Q)$ with $h_2(F)$. 

\begin{proposition}\cite[Proposition 7.1]{bk}
\label{2rankProp}
Let $E/\Q$ be an elliptic curve with $E(\Q)[2]=0$. Then
 \begin{small}\begin{equation}
\label{2rankbound}   
\text{dim}_{\F_2}\ S_2(E/\Q) \leq h_2(F)+ u+ e+ \sum\limits_{p \in \Phi_{a}}(n_{p}-1). 
\end{equation} \end{small}
Here $u =1 $ if the discriminant of $E$ over $\Q$, $\Delta(E)<0$ and $u=2$, if  $\Delta(E)>0$.  Next, $e$ denotes  the cardinality of certain specified subset of rational primes  where $E$ has multiplicative reduction. Further, 
 $\Phi_{a}$ is the  set of rational primes at which $E$ has additive reduction and 
 $n_{p}$ denotes the number of primes lying over $p$ in the ring of integers of $F$. \qed
\end{proposition}

\begin{rem}
\label{Eccoef}
{\rm
 The  $j$-invariant of the elliptic curve $E_{-432n^2}$ vanishes and  hence it does not have multiplicative reduction at any rational prime. Also, the discriminant of $E_{-432n^2}$ over $\Q$ is given by $\Delta(E_{-432n^2})= -2^{12}3^{9}n^4$. Thus  $E_{-432n^2}$ has good reduction at every prime $q\nmid 6n$. Further, it has additive reduction at $3$ and if $n$ is cube-free, then it has additive reduction at every prime $q\mid n$. When $n$ is odd,  it is easy to see that $2$ is prime of good reduction of $E_{-432n^2}/\Q$. 
}
\end{rem}

Now we apply Proposition \ref{2rankProp} to our curve  $E_{-432n^2}$ to complete the proof of Proposition \hyperlink{thm:C}{C}.


\begin{proof}[Proof of Proposition \texorpdfstring{\hyperlink{thm:C}{C}}]
By our assumption in Proposition \hyperlink{thm:C}{C}, $n >2 $ is a cube-free integer with  $\operatorname{cf}(4n) \not\equiv \pm 1 \pmod{9}$. Since $n>2$ is a cube-free integer, it is immediate that $E_{-432n^2}(\Q)[2]=0$. Thus, for  $E_{-432n^2}: y^2=f(x)=x^3- 432n^2$, we can apply Proposition \ref{2rankProp} by taking $F=\Q(\sqrt[3]{4n})$. By Remark \ref{Eccoef}, the discriminant $\Delta(E_{-432n^2})$ is negative and $E_{-432n^2}$ does not have multiplicative reduction at any rational prime. Further, $E_{-432n^2}$ has additive reduction at $3$ and at every prime dividing $n$. Thus  \eqref{2rankbound} reduces to 
 \begin{small}
\begin{equation}
\label{2rankbound-refined}   
\text{dim}_{\F_2}\ S_2(E_{-432n^2}/\Q) \leq h_2(F)+ 1 + \sum\limits_{p \mid 3n}(n_{p}-1). 
\end{equation}
\end{small}
\noindent Further, using the hypothesis $\operatorname{cf}(4n) \not\equiv \pm 1 \pmod{9}$, we can deduce from  Lemma \ref{ramification} that  $ n_{p} =1$ holds for each integer prime $p\mid 3n$. Thus   \eqref{2rankbound-refined} further reduces to
 \begin{small}
\begin{equation}
\label{2rankbound-final}   
\text{dim}_{\F_2}\ S_2(E_{-432n^2}/\Q) \leq h_2(4n)+ 1. 
\end{equation}
\end{small}
 By our assumption, the global root number of $E_{-432n^2}/\Q$, $w(n)=1$. By applying Theorem \ref{s2w}, we obtain that $\mathrm{dim}_{\Q_2}\ \mathrm{Hom}_{\Z_2}\big(S_{2^\infty}(E_{-432n^2}/\Q),\Q_2/{\Z_2}\big)\otimes \Q_2$ is even. As $E_{-432n^2}(\Q)[2]=0$, we have from Lemma \ref{2and2infinity} that $\text{dim}_{\F_2}\ S_2(E_{-432n^2}/\Q)$ is even as well. Further, as  $ n $ is given to be a rational cube sum i.e. $\text{rank}_\Z \ E_{-432n^2}(\Q) >0$ it follows that $ \text{dim}_{\F_2}\ S_2(E_{-432n^2}/\Q) $ is positive. consequently,  $ \text{dim}_{\F_2}\ S_2(E_{-432n^2}/\Q)$ is a positive  even integer and we get from \eqref{2rankbound-final} that
$h_2(4n) \geq 1$. This completes the proof of Proposition \hyperlink{thm:C}{C}. 
\end{proof}

We begin the preparation for the proof of Theorems  \hyperlink{thm:A}{A} \& \hyperlink{thm:D}{D}  with a couple of lemmas.
\begin{lemma}
\label{lem2l=1,3rank}
Suppose that  $\ell$ is a prime with $\ell \equiv 1\pmod 9$. Then the $3$-rank of ideal class group of $F=\mathbb{Q}(\sqrt[3]{2\ell})$ is at least one i.e. $h_3(2\ell) \geq 1$. Moreover, $h_3(2\ell)=2$ if and only if $\legendre[3]{2}{\ell}=1$. 
\end{lemma}
\begin{proof}
We have $ n = 2\ell $ with the prime $ \ell \equiv 1 \pmod{9} $. From the equation \eqref{n}, we obtain $w = v=1$. From \eqref{x}, we can set $ x(n) = (x_1, x_2) = (\pi\pi^{\prime 2}, \ell) $, where $ \ell = \pi \pi' $ represents the prime factorization of $ \ell $ in $\mathcal{O}_K=\Z[\zeta] $.  
Further, from \eqref{B}, $B$ is a $ 1 \times 4$ matrix over $\F_3$ and by Lemma \ref{cubic3rank}, it is immediate that $h_3(n) \geq 1$, proving the first part of the result.  Again from \eqref{B}, the entries $ \beta_{1j} \in \F_3$ of $B$, where $1 \leq j \leq 4$,  are determined as follows:
\begin{small}
\begin{align*}
 \zeta^{\beta_{11}} = \Hilbert{x_{2}}{n}{\pi}=\Hilbert{\ell}{2\ell}{\pi} , \quad  
 \zeta^{\beta_{12}} = \Hilbert{x_{2}}{n}{\pi'}=\Hilbert{\ell}{2\ell}{\pi'}, \quad 
 \zeta^{\beta_{13}} = \Hilbert{x_{2}}{n}{2}=\Hilbert{\ell}{2\ell}{2}, \quad \zeta^{\beta_{14}} = \Hilbert{x_2}{\mathfrak{p}}{\mathfrak{p}}=\Hilbert{\ell}{\mathfrak{p}}{\mathfrak{p}} 
\end{align*}
\end{small}

We will compute the Hilbert symbols $\Hilbert{\ell}{2\ell}{\pi}$, $\Hilbert{\ell}{2\ell}{\pi'}$, $\Hilbert{\ell}{2\ell}{2}$ and $\Hilbert{\ell}{\mathfrak{p}}{\mathfrak{p}}$   individually. Since $\ell = \pi \pi'$ in $\mathbb{Z}[\zeta]$, it follows that $v_{\pi}(\ell) = v_{\pi}(2\ell) = 1$. Therefore, applying \eqref{Hilbertsymbol}, we obtain \begin{small}{$$\Hilbert{\ell}{2\ell}{\pi} = \legendre[3]{1/2}{\pi} = \legendre[3]{4}{\pi}.
$$}\end{small}
Similarly, we have  
$
\Hilbert{\ell}{2\ell}{\pi'} = \legendre[3]{4}{\pi'}.
$  
Moreover, it is known that $\legendre[3]{4}{\pi'} = \legendre[3]{4}{\pi}^{-1}$ (see \cite[chapter 7]{lem}), which implies that $\beta_{11} = -\beta_{12}$ in $\mathbb{F}_3$. Next, using \eqref{Hilbertsymbol}, we compute  
\begin{small}{\begin{equation*}
 \Hilbert{\ell}{2\ell}{2}=\legendre[3]{\ell}{2}=\legendre[3]{\pi \pi'}{2} = \legendre[3]{\pi}{2}\legendre[3]{\pi'}{2} = \legendre[3]{2}{\pi} \legendre[3]{2}{\pi'}.
\end{equation*}}\end{small}
Here, the last equality follows from the law of cubic reciprocity. Since $\legendre[3]{2}{\pi}^{-1} = \legendre[3]{2}{\pi'}$, it follows that $\Hilbert{\ell}{2\ell}{2} = 1.$  Thus, we conclude that $\beta_{13} = 0$ in $\mathbb{F}_3$. Apart from that, since $\ell \equiv 1 \pmod{9}$, we deduce $\Hilbert{\ell}{\mathfrak{p}}{\mathfrak{p}} = 1$ and hence $\beta_{14}=0$. From Lemma \ref{cubic3rank}, we have $h_3(n) = 2 - \text{rank } B$. consequently, $h_3(n) = 2$ if and only if $B$ is the $1 \times 4$ zero matrix over $\mathbb{F}_3$. This occurs precisely when $\beta_{11} = 0$ in $\mathbb{F}_3$, which implies\begin{small}{$$\zeta^{\beta_{11}} = \Hilbert{\ell}{2\ell}{\pi} = \legendre[3]{4}{\pi} = \legendre[3]{2}{\pi}^2= 1.$$}\end{small}
It follows that $h_3(n) = 2 \Leftrightarrow \legendre[3]{2}{\pi} = 1$. Observe that as $\ell$ splits as $ \ell = \pi \pi' $  in $\mathcal{O}_K$, we have $\frac{\Z}{\ell\Z} \cong \frac{\OO_K}{\pi \OO_K}$. 
Thus $\legendre[3]{2}{\pi} = 1 \Leftrightarrow \legendre[3]{2}{\ell} = 1$. This completes the proof of the lemma.  
\end{proof}

\begin{lemma}
\label{lem3l=7,4,3rank}
Let $ n = 12\ell $ (or $ n = 18\ell $, respectively), where $ \ell $ is a prime with $ \ell \equiv 7 \pmod 9 $ (or $\ell \equiv 4 \pmod 9 $, respectively). Then, the 3-rank of the ideal class group of the cubic field $ F = \mathbb{Q}(\sqrt[3]{n}) $ is at least 1. Moreover,  $ h_3(n)=2 $ if and only if $ \legendre[3]{3}{\ell} = 1$.
\end{lemma}
\begin{proof}
We proceed in a similar way as in  Lemma \ref{lem2l=1,3rank}. Suppose that $ n = 12\ell $ with the prime $ \ell \equiv 7 \pmod{9} $.  From  \eqref{n}, we obtain $w =1$  and $v=0$. Next consider $ x(n) = (x_1, x_2) = (\pi\pi^{\prime2}, 2^{2}\ell) $, where $ \ell $ splits in $\Z[\zeta] $ as $\ell=\pi \pi' $.  By \eqref{B}, we get that  $B$ is a $ 1 \times 4 $ matrix over $\F_3$. As before, applying Lemma \ref{cubic3rank}, we deduce that $h_3(n) = 2$ if and only if $\beta_{1j}=0$ in $\mathbb{F}_3$ for $1\leq j \leq 4$, which occurs precisely when $\beta_{11} = 0 \Leftrightarrow \legendre[3]{3}{\ell} = 1$, as required. The proof in the other case is also similar. \end{proof}
The cubic residue symbol in the above lemmas plays an important role in determining whether the corresponding integer is cube sum.
\begin{proposition}
 \label{prop2l=1cubic}  
 Let $\ell \equiv 1 \pmod{9}$ be a prime, and suppose $n \in \{2\ell, 2\ell^2\}$. If $n$ can be expressed as a sum of two rational cubes, then $\left(\frac{2}{\ell}\right)_3 = 1$.
\end{proposition}
\begin{proof}
Let $\ell$ be a prime with $\ell \equiv 1\pmod 9$. It is proved in \cite[Theorem 1.2]{jms} that if  $\legendre[3]{2}{\ell} \neq 1$, then both $2\ell$ and $2\ell^2$ are non-cube sums. The curve $E_{-432n^2}$ has a $3$-isogeny $\varphi_n$ over $\Q$  and the main idea of the proof of \cite[Theorem 1.2]{jms} is to explicitly compute  $S_{\varphi_n}(E_{-432n^2}/K)$ when $\legendre[3]{2}{\ell} \neq 1$. 
\end{proof}

As mentioned in the introduction, the elliptic curve $E_{-432n^2}/\Q$ in general has bad, additive reduction at $3$, which makes $3$-descent  more difficult. Further,  it seems there is a scarcity of literature for the cube sum problem  in the case where  $3\mid n$, the root number $w(n)$ is $1$ with potentially positive rank of $E_{-432n^2}(\Q)$. We discuss the set up before going in to the proof of Theorem \hyperlink{thm:A}{A}. 
Recall  that $K=\Q(\zeta)$ and $\mathfrak p=1-\zeta.$ Let $\Sigma_K$ denote the set of all finite places of $K$. For a finite subset $S$  of $\Sigma_K$, $\OO_S=\OO_{K,S}$ denotes the set of $S$-integers of $K$. A general element of $\Sigma_K$  will be denoted by $\q$. 
 Let  $\OO_\q$ be the ring of integers of $K_\q$ and for $T \in \{\OO_S, K, \OO_\q\}$,  
let $N:T^* \times T^* \to T^*$ denotes the `norm' map sending $(x,y) \to xy$ for all $x, y \in T^*$ and set 
\begin{small}$\Big(\frac{T^*}{T^{*3}} \times \frac{T^*}{T^{*3}}\Big)_{N=1}=\text{ker}(\bar{N}):=\Big\{ (\bar{x},  \bar{y}) \in \frac{T^*}{T^{*3}} \times \frac{T^*}{T^{*3}} \mid \bar{x} \bar{y}=\bar{1}\Big\}.$\end{small} It is plain that $\Big(\frac{T^*}{T^{*3}} \times \frac{T^*}{T^{*3}}\Big)_{N=1}\cong \frac{T^*}{T^{*3}}$. Following this isomorphism, to ease the notation, we identify $\Big(\frac{T^*}{T^{*3}} \times \frac{T^*}{T^{*3}}\Big)_{N=1}$ with $ \frac{T^*}{T^{*3}}$ and simply denote an element $(\bar{x}, \bar{x}^2) \in \Big(\frac{T^*}{T^{*3}} \times \frac{T^*}{T^{*3}}\Big)_{N=1}$ by $\bar{x} \in \frac{T^*}{T^{*3}}$.

For any positive integer $n$ consider the elliptic curve $E_{-432n^2}$. We have  degree-$3$ rational  isogenies  $
E_{-432n^2} 
\quad
\mathrel{\substack{\xrightarrow{\varphi_n} \\[-0.7ex] \xleftarrow[\hat{\varphi}_n]{} }}
\quad
E_{16n^2}.$ 
Further, as $-3\in {K^*}^2$, the curves $E_{-432n^2} $ and ${E}_{16n^2}$ are isomorphic over $K$. So we get a $3$-isogeny  over $K$, $\phi_n: E_{-432n^2} \lra E_{-432n^2}$,   given by (see \cite[equation (1)]{jms}): \begin{small}\begin{equation}\label{eq:defofphi}
 \phi_n(x,y)= \left( \frac{x^3+4\cdot(-432n^2)}{\p^2x^2}, \frac{y\big(x^3-8\cdot(-432n^2)\big)}{\p^3x^3} \right).
\end{equation}\end{small} 
Recall that the Kummer map $\delta_{\phi_n, K_q}$ is defined in \S \ref{subsecselmer}. Now the Selmer group $S_{\phi_n}(E_{-432n^2}/K)$ can be explicitly written as follows (see \cite[\S1]{jms}): 
\begin{small} \begin{equation}\label{eq:newseldefsq}
{S}_{\phi_n}(E_{-432n^2}/K)=\{ \overline{x} \in K^*/K^{*3}  \mid \overline{x}\in \text{Image} (\delta_{\phi_n,K_\q}) \text{ for all } \q \in \Sigma_K \}.
\end{equation}\end{small}
\begin{small}  \begin{equation}\label{sn}
\text{Put } 
    S_n  :=\{ \q \in \Sigma_K  \mid \upsilon_\q(4\cdot432n^2) \not\equiv 0 \pmod 6 \}.
 \end{equation}\end{small}
 It follows from \cite[Theorems 3.15 \&  4.14(2)]{jms} that \begin{small}${S}_{\phi_n}(E_{-432n^2}/K) \subset \frac{\OO_{S_n}^*}{\OO_{S_n}^{*3}}
$\end{small} and\\ \begin{small}$ \operatorname{dim}_{\F_3} \ {S}_{\phi_n}(E_{-432n^2}/K) \le \#S_n+1$.\end{small} 
In particular, an element \begin{small}$\overline{x} \in \frac{\OO_{S_n}^*}{\OO_{S_n}^{*3}} $\end{small}  is in \begin{small}${ S}_{\phi_n}(E_{-432n^2}/K) $\end{small}  if and only if $\overline{x} \in \text{Image} (\delta_{\phi_n,K_\q})$ for all  $\q \in \Sigma_K$. With these set up, we can now prove Proposition \ref{prop3l=7,4,cubic}:

\begin{proposition}
\label{prop3l=7,4,cubic}
Let $ n = 3\ell $ (or $ n = 3\ell^2 $, respectively), where $ \ell $ is a prime with $ \ell \equiv 7 \pmod 9$ (or $\ell \equiv 4 \pmod 9 $, respectively). If $n$ is a rational cube sum, then  $\legendre[3]{3}{\ell}=1$.
\end{proposition}
\begin{proof}
We consider the case $n=3\ell,$ where $\ell$ is a prime with $\ell \equiv 7 \pmod 9$. The proof for $n=3\ell^2$ with $\ell \equiv 4 \pmod 9$ is similar. We proved the contrapositive statement i.e. if  $3$ is not a cube modulo $\ell$, then we show that  $\mathrm{rank}_\Z \  E_{-432(3\ell)^2}(\Q) =0$.

We have the rational  $3$-isogenies 
 $
E_{-432(3\ell)^2} 
\quad
\mathrel{\substack{\xrightarrow{\varphi_\ell} \\[-0.7ex] \xleftarrow[\hat{\varphi}_\ell]{} }}
\quad
E_{(12\ell)^2}
$
and these two curves are isomorphic over $K$. In particular, $\mathrm{rank}_\Z \  E_{-432(3\ell)^2}(\Q)  =\mathrm{rank}_\Z \ E_{(12\ell)^2}(\Q)$. Further, note that $E_{-432(3\ell)^2}(\Q)[3]=0$ and $E_{(12\ell)^2}(\Q)[\hat{\varphi}_\ell]\cong \frac{\Z}{3\Z}$.  Setting $R:=\frac{\Sh({E}_{(12\ell)^2}/\Q)[\widehat{\varphi}_\ell]}{{\varphi}_\ell(\Sh({E}_{-432(3\ell)^2}/\Q)[3])}$, 
it follows from \cite[Lemma 6.1]{ss} that  
\begin{small}
\begin{equation}\label{dimevenphi}
    \begin{split}
        \dim_{\F_3} {S}_3({E}_{-432(3\ell)^2}/\Q)  &= \dim_{\F_3} {S}_{\phi_\ell}({E}_{-432(3\ell)^2}/K)-\dim_{\F_3} R - \dim_{\F_3} \frac{{E}_{(12\ell)^2}(\Q)[\hat{\varphi}_\ell]}{\varphi_\ell(E_{-432(3\ell)^2}(\Q)[3])}\\
          &= \dim_{\F_3} {S}_{\phi_\ell}({E}_{(12\ell)^2}/K)-\dim_{\F_3} R - 1\\
    \end{split}
\end{equation}
\end{small}
Further, Cassels-Tate pairing induces a non-degenerate, alternating pairing on $R$, so that $\dim_{\F_3} R$ is even (see \cite[Proposition 49]{bes}). From \eqref{rootnumber}, we get that the global root number of $E_{-432(3\ell)^2}$ over $\Q$,  $w(3\ell)=1$. Thus by applying the $p$-parity result in Theorem \ref{s2w} for $p=3$, we deduce that $\mathrm{dim}_{\Q_3}\ \mathrm{Hom}_{\Z_3}\big(S_{3^\infty}(E_{-432(3\ell)^2}/\Q),\Q_3/{\Z_3}\big)\otimes_{\Z_3} \Q_3$ is even. Moreover, as $E_{-432(3\ell)^2}(\Q)[3]=0$, we obtain  from Lemma \ref{2and2infinity}  that $\dim_{\F_3} {S}_3({E}_{-432(3\ell)^2}/\Q)$ is even as well. Then it is immediate from \eqref{dimevenphi} that $\dim_{\F_3} {S}_{\phi_\ell}({E}_{(12\ell)^2}/K)$ is odd.

Now we claim that:
\begin{small}\begin{equation}\label{mainclaim}
\text{under the assumption } \ {\legendre[3]{3}{\ell}}\neq 1, \text{ we have } \ \mathrm{dim}_{\F_3}\ {S}_{\phi_\ell}(E_{(12\ell)^2}/K) \leq 2.
\end{equation}\end{small}
Assume the claim at the moment. Then it follows from the above discussion that $\mathrm{dim}_{\F_3}\ {S}_{\phi_\ell}(E_{(12\ell)^2}/K) $ must be equal to $1$ and consequently, we deduce from \eqref{dimevenphi} that ${S}_3({E}_{-432(3\ell)^2}/\Q) =0$. 
Then it is plain from \eqref{n-selmer-d} that $\mathrm{rank}_\Z \ E_{-432(3\ell)^2}(\Q)=0$ and hence $n=3\ell$ is a non-cube sum. Thus, it suffices to establish \eqref{mainclaim} to complete the proof of the theorem and in the rest of the proof, we establish \eqref{mainclaim}.

To ease the notation, for the rest of the proof, we write $t=(12\ell)^2$ and put $\phi=\phi_\ell$. Also recall that $\ell$ splits in $\OO_K$ as $\ell=\pi{\pi'}$. Then in the above setting, with $S_{t}=\{\p=1-\zeta, \pi,{\pi'}\}$, we have $\OO^*_{S_t}=\langle \pm \zeta, \p, \pi, {\pi'} \rangle$ and\begin{small}$$S_\phi(E_{t}/K) \subset \frac{\OO_{S_t}^*}{\OO_{S_t}^{*3}} =\langle\overline{\zeta}^2, \  \overline{9}, \  \overline{\pi}^2,  \  \overline{9\ell^2}\rangle.$$\end{small} We see that $\dim_{\F_3} {S}_\phi(E_{t}/K) \le 4$. 
Moreover, by the formula of the Kummer map given in \cite[\S14, \S15]{cass}, we deduce  that $1/\overline{24\ell}  = \overline{9\ell^2} \in {S}_\phi(E_{t}/K)$, being the image of $(0, \ 12\ell)$ under $\delta_{\phi,K_\q}$ for all $\q$. Also $\overline{9\ell^2}$ is a non-zero element of ${S}_\phi(E_{t}/K)$. Next, by \cite[Prop. 4.6(b)]{jms},  for a prime $\q \nmid 3$ in $\OO_K$ with $v_{\q}(4\cdot 144\ell^2) \not\equiv 0 \pmod{6}$, we have that $\delta_{\phi, K_\q}(E_t(K_\q)) \cap \frac{\OO_{\q}^*}{\OO_{\q}^{*3}}   =\{1\}.$ Applying this for $\q=\pi$ and observing that $\ell \equiv 7 \pmod 9$, we deduce that $\overline{\zeta}_3^2 \notin \text{Image}( \delta_{\phi,K_\pi})$ and hence it is not an element of ${S}_\phi(E_{t}/K)$.  consequently,   $1 \le \dim_{\F_3} \ {S}_\phi(E_{t}/K) \le 3$. We will now go on to show that $\dim_{\F_3} \ { S}_\phi(E_{t}/K) \leq 2$, as required.

By our assumption, we have $\ell \equiv 7 \pmod 9$ and 
$\legendre[3]{3}{\ell}\neq 1$. We also have $\overline{9\ell}^2 \in { S}_\phi(E_{t}/K)$. There are $13$ distinct subgroups of $\frac{\OO_{S_t}^*}{\OO_{S_t}^{*3}} $  of order $9$ containing $\overline{9\ell}^2 $.
In fact,  explicitly  the $13$ generators in $\frac{\OO_{S_t}^*}{\OO_{S_t}^{*3}}$ corresponding to order $3$ subgroups of $\frac{\OO_{S_t}^*}{\OO_{S_t}^{*3}}/{ \big\langle \overline{9\ell}^2 \big\rangle}$ are  given by   $\overline{\zeta}^2$, \ $\overline{9}, $ \ $\overline{\pi}^2$, \ $\overline{9\zeta^2}$, \ $\overline{3\zeta^2}$, \ $\overline{9\pi^2}$, \ $\overline{9\pi}$, \ $\overline{\zeta^2\pi^2}$, \ $\overline{\zeta\pi^2}$, \ $\overline{9\zeta^2\pi^2}$, \ $\overline{3\zeta^2\pi^2}$, \ $\overline{9\zeta\pi^2}$ and $\overline{3\zeta\pi^2}$. 
We consider  these $13$ generators and for  each of them, produce a prime $\q$ such that it does not lie in $\operatorname{Image}(\delta_{\phi, K_{\q}})$  and this rules out the possibility that it is an element of ${ S}_\phi(E_{t}/K)$.

We assume that $\left(\frac{3}{\pi}\right)_3 = \zeta$; the case $\left(\frac{3}{\pi}\right)_3 = \zeta^2$ can be handled similarly. Note that, for $\ell \equiv 7 \pmod{9}$,  we have that $\legendre[3]{\zeta}{\pi}=\legendre[3]{\zeta}{\pi'}=\zeta^{2}$ (see \cite[\S 2, chapter 7]{lem}).

We have already noticed  $\overline{\zeta}^2 \notin \text{Image}( \delta_{\phi,K_\pi})$.  Next, we consider  $\overline{9}.$ From  \cite[Prop. 4.6(2)]{jms}, we know that $\delta_{\phi, K_\pi}(E_t(K_\pi)) \cap \frac{\OO_{\pi}^*}{\OO_{\pi}^{*3}}  =\{1\}.$  As  ${\legendre[3]{3}{\pi}}\neq 1$, we get that $\overline{9} \notin \text{Image}(\delta_{\phi,K_\pi}).$ Now we show  $\overline{\pi}^2 \notin { S}_\phi(E_{t}/K)$. Indeed, as $\overline{9\ell}^2 \in { S}_\phi(E_{t}/K)$, if we assume that $\overline{\pi}^2 \in { S}_\phi(E_{t}/K)$, then it will imply that $\overline{9{\pi'}}^2\in { S}_\phi(E_{t}/K)$. However, by Evans' trick \cite[\S7] {lem}, $\legendre[3]{\pi}{\pi'}=1$ and we also have $\legendre[3]{3}{\pi'}=\zeta^2 \neq 1$ and hence  $\overline{9{\pi'}^2} \notin \text{Image}( \delta_{\phi,K_\pi})$, a contradiction. Proceeding in a similar way, we can show that none of $ \overline{3\zeta^2}, \;\overline{9\zeta\pi^2}$ are in the image of  $\delta_{\phi,K_{\pi}}$. On the other hand, none of  $\overline{9\zeta^2}$, $\overline{9\pi^2}$, \ $\overline{9\pi}$, \ $\overline{\zeta^2\pi^2}$, \ $\overline{\zeta\pi^2}$, $\overline{9\zeta^2\pi^2}$, $\overline{3\zeta\pi^2}$ are in the image of $\delta_{\phi,K_{\pi'}}$. 

 Thus other than $\overline{9\ell}^2$,   the  only possible element in ${ S}_\phi(E_{t}/K)$ is  $\overline{3\zeta^2\pi^2}$, whence \\ $\mathrm{dim}_{\F_3}\ {S}_{\phi}(E_{(12\ell)^2}/K) \leq 2$, which establishes \eqref{mainclaim}. This completes the proof of the theorem.\end{proof}
\begin{rem}\label{lastrem}{\rm 
  If we compute $\operatorname{Image}(\delta_{\phi_\ell,K_{\p}})$ explicitly, then we can show that   $\overline{3\zeta^2\pi^2} \notin {S}_{\phi_\ell}(E_{(12\ell)^2}/K)$ and hence $\mathrm{dim}_{\F_3}\ {S}_{\phi_\ell}(E_{(12\ell)^2}/K) =1 $. However, we could get around this explicit calculation in Proposition \ref{prop3l=7,4,cubic} by using the $3$-parity result.

  If $3\mid n$, then $\p \in S_n$ (see \eqref{sn}) 
 and the image of the Kummer map  at $\p$ for $E_{-432n^2}$ is difficult to compute directly (see \cite[Remark 4.13]{jms}). However, notice that there are only finitely many elliptic curves $E_{-432n^2}$ over $K_\mathfrak p$, up to isomorphism, depending only on the cube-class of $n$. Hence it would be enough to compute the image of the Kummer map for any one of the (suitably chosen) member in each of the finitely many cube-classes. We thank the referee for explaining this alternative method to us.} 
\end{rem}
Now we can complete the proofs of Theorems \hyperlink{thm:A}{A} and \hyperlink{thm:D}{D}.
\begin{proof}[Proof of Theorem \hyperlink{thm:D}{D}]

The first part statement of Theorem \hyperlink{thm:D}{D}  follows from Lemma  \ref{lem2l=1,3rank} and  Proposition \ref{prop2l=1cubic}.

For the second part, recall by Lemma  \ref{lem2l=1,3rank},  for a prime $\ell \equiv 1 \pmod 9$, $h_3(2\ell)=2 \Leftrightarrow  \legendre[3]{2}{\ell} = 1 $. Further, by Proposition \ref{prop2l=1cubic}, we know that if  $\ell \equiv 1 \pmod 9 \text{ and } \legendre[3]{2}{\ell} \neq 1 $, then both $2\ell$ and $2\ell^2$ are non cube-sums. It is a classical result (see \cite[Page 55]{cox}) that \begin{small}$$ \{\ell \text{  prime} : \legendre[3]{2}{\ell} \neq 1\}=\{\ell \text{  prime} : \ell= 4x^2 -2xy +7y^2, \text{ for some } x, y \in \Z\}.$$\end{small}   
Thus it suffices to show that 
\begin{small}$$S:= \{\ell \text{  prime} : \ell \equiv 1\pmod 9 \text{ and } \ell= 4x^2 -2xy +7y^2, \text{ for some } x, y \in \Z\}$$\end{small}
has a positive Dirichlet density. 
Note that the binary quadratic form $4X^2 -2XY +7Y^2 \in \Z[X,Y]$ has discriminant $= -108$ and it represents the prime $19 \equiv 1 \pmod 9$ at $(X,Y)=(2,1)$ and $19\nmid 108$. Then it follows from \cite[Proposition 1, Part (1)]{hal} (which extends the work of  \cite{mey})  that the set $S$ above has a positive Dirichlet density. The completes the proof of  Theorem  \hyperlink{thm:D}{D}.
\end{proof}
\begin{proof}[Proof of Theorem \hyperlink{thm:A}{A}]
The first part of Theorem 
    \hyperlink{thm:A}{A} is immediate from  Lemma  \ref{lem3l=7,4,3rank} and Proposition \ref{prop3l=7,4,cubic}. For the density results in the second part, we only give a proof for $3\ell$ with $\ell \equiv 7 \pmod 9$ and the proof in the other case is similar. 

Observe that the Galois group of $x^3-3$ over $\Q$ is $S_3$ and applying the Chebotarev density theorem, we can get that  density of the set $\{\ell \text{ prime} : \legendre[3]{3}{\ell} = 1\}$ is $2/3.$ For a prime $q\equiv 2 \pmod 3$, every integer is a cube in $\F_q$, so   the density of the set $\{\ell \text{  prime} : \ell \equiv 1 \pmod 3 \text{ and } \legendre[3]{3}{\ell} = 1\}$ is $1/6$. Now it is well known that 
\begin{small}\begin{equation}\label{3cubemodl2}
    \begin{split}
    \{\ell \text{  prime} : \ell \equiv 1 \pmod 3 \text{ and } \legendre[3]{3}{\ell} = 1 \} & = \{\ell \text{  prime} : 4\ell= x^2 +243y^2, \text{ for some } x, y \in \Z\}\\ &=  \{\ell \text{  prime} : \ell= x^2 +xy +61y^2, \text{ for some } x, y \in \Z\}.
\end{split} 
\end{equation}\end{small}
 The binary quadratic form $X^2 +XY +61Y^2 \in \Z[X,Y]$  has discriminant $=-243$ and it represents the primes $61\equiv 7 \pmod 9$, $67\equiv 4 \pmod 9$ and $73 \equiv 1 \pmod 9$. Thus, we can again deduce using \cite[Proposition 1, Part (1)]{hal} that for each  $k\in\{1,4,7\}$,  the set
\begin{small}$$P_k:= \{\ell \text{  prime} : \ell \equiv k\pmod 9 \text{ and } \ell= x^2 +xy +61y^2, \text{ for some } x, y \in \Z\}$$\end{small} has  positive Dirichlet density. In particular, Dirchilet density of $P_7$ is positive but strictly less than $1/6$. Hence we can conclude from \eqref{3cubemodl2} that  $ \{\ell \text{  prime} : \ell \equiv 7\pmod 9 \text{ and } \legendre[3]{3}{\ell} \neq 1 \}$ has a positive Dirichlet density, as required. 
\end{proof}
\begin{proof}[Proof of Corollary \hyperlink{thm:B}{B}]
We apply Proposition \hyperlink{thm:C}{C} with $n=3\ell$ and $\ell\equiv 7 \pmod 9$ is a prime. Note that $\mathrm{cf}(4n)=12\ell \equiv 3 \pmod 9$. Then we deduce by the same proposition that the class number of  $\mathbb{Q}(\sqrt[3]{12\ell})$ is even. Now the assertion (i) of the corollary follows from Theorem \hyperlink{thm:A}{A}. The proof for the second case is similar (observe that $\Q(\sqrt[3]{12\ell^2})=\Q(\sqrt[3]{18\ell})$).
\end{proof}
\section{Numerical Examples:}\label{nexampl}
In this section, we discuss various numerical examples related to our results. We begin with  numerical example which shows that the criteria obtained in  Theorems \hyperlink{thm:A}{A}, \hyperlink{thm:D}{D} and Proposition \hyperlink{thm:C}{C} are  necessary but not sufficient.
\begin{example}\label{ex1}
Theorem \hyperlink{thm:A}{A}:  Let $ \ell = 547 \equiv 7 \pmod 9$. It can be verified that $\operatorname{Cl}_{\Q(\sqrt[3]{12\ell})} \cong \Z/3\Z \oplus \Z/3\Z$,  but $ 3\ell $ is a non-cube sum. For $ \ell = 67 $, we have $\operatorname{Cl}_{\Q(\sqrt[3]{18\ell})} \cong \Z/3\Z \oplus \Z/3\Z$,  although $ 3\ell^2 $ is a non-cube sum.
 
Theorem \hyperlink{thm:D}{D} : Let $ \ell = 919 \equiv 1 \pmod 9$. We have $\operatorname{Cl}_{\Q(\sqrt[3]{2\ell})} \cong \Z/3\Z \oplus \Z/9\Z$, i.e. $h_{3}(2\ell) = 2$,  even though $ 2\ell $ is a non-cube sum. For $ \ell = 109 $, we can check that, $\operatorname{Cl}_{\Q(\sqrt[3]{2\ell})} \cong \Z/3\Z \oplus \Z/6\Z$,  although $ 2\ell^2 $ is not a sum of two rational cubes.

Proposition \hyperlink{thm:C}{C}: Let $ \ell = 739 \equiv 1 \pmod 9$. We compute $\operatorname{Cl}_{\Q(\sqrt[3]{4\ell})} \cong \frac{\Z}{3\Z} \oplus \frac{\Z}{6\Z}$, i.e. $h_{2}(4\ell) = 1$, even though $ \ell $ is a non-cube sum. Similarly, for $ \ell = 199 $, we verify that $\operatorname{Cl}_{\Q(\sqrt[3]{2\ell})} \cong \Z/6\Z$,  although $ \ell^2 $ is a non-cube sum.
\end{example}
Next, we point out  that both the assumptions (i) and (ii) are needed in  Proposition \hyperlink{thm:C}{C}. 
\begin{example}\label{ex2}
   Consider $ n = 254 = 2 \cdot 127 $. Observe that by \eqref{rootnumber}, $w(n)= 1 $. We can check that $n$ is a cube sum. Notice that $\operatorname{cf}(4n) = 127 \equiv 1 \pmod{9}$, so hypothesis (i) fails and we have $\operatorname{Cl}_{\Q(\sqrt[3]{127})} \cong \Z/{3\Z}$. 
On the other hand, $ n = 13 $ is a cube sum. In this case,  $w(n)=-1$, so condition (ii) does not hold and we get \begin{small}$\operatorname{Cl}_{\Q(\sqrt[3]{52})} \cong \Z/{3\Z}$.\end{small} 
\end{example}

We  demonstrate our results in   Theorems \hyperlink{thm:A}{A}, \hyperlink{thm:D}{D} and Proposition \hyperlink{thm:C}{C} through numerical examples of cube sum and non-cube sum integers, computed via \cite{sage},  in Table \ref{tab:class_numbers}.

\newpage

\begin{table}[H]
    \centering
    \renewcommand{\arraystretch}{1} 
    \setlength{\tabcolsep}{10pt} 
    \caption{class numbers and ranks for different values of $\ell$}
    \label{tab:class_numbers}
    \begin{tabular}{|c|c|c||c|c|c|}
        \hline
        \multicolumn{3}{|c||}{Proposition \hyperlink{thm:C}{C}, $n=\ell$, $\ell \equiv 1\pmod{9}$} & \multicolumn{3}{c|}{Proposition \hyperlink{thm:C}{C}, $n=\ell^{2}$, $\ell \equiv 1\pmod{9}$} \\
        \hline
        $\ell$ & $r_{\operatorname{al}}(\ell)$ & $h(4\ell)$ & $\ell$  & $r_{\operatorname{al}}(\ell^2)$ & $h(2\ell)$ \\
        \hline
        19  &2  & 6      & 109  & 2 & 18      \\
        37  &2  & 6      & 181  & 2 & 12      \\
        127 &2  & 18     & 271  & 2 & 6       \\
        163 &2  & 12     & 739  & 2 & 36      \\
        271 &2  & 6      & 2503 & 2 & 12      \\
        379 &2  & 24     & 2521 & 2 & 12      \\
        397 &2  & 108    & 2953 & 2 & 18      \\
        \hline
        73 & 0 & 3       & 19   & 0 & 3       \\
        109& 0 & 3       & 37   & 0 & 3       \\
        \hline
    \end{tabular}
    
    \begin{tabular}{|c|c|c|c||c|c|c|c|}
        \hline
        \multicolumn{4}{|c||}{Theorem \hyperlink{thm:D}{D}, $n=2\ell$, $\ell \equiv 1\pmod{9}$} & \multicolumn{4}{|c|}{Theorem  \hyperlink{thm:D}{D}, $n=2\ell^{2}$, $\ell \equiv 1\pmod{9}$} \\
        \hline
        $\ell$ & $r_{\operatorname{al}}(2\ell)$  & $h(2\ell)$& $h_{3}(2 \ell)$ & $\ell$ & $r_{\operatorname{al}}(2\ell^2)$ & $h(2\ell)$ & $h_{3}(2 \ell)$ \\
        \hline
        109  & 2 & 18     & 2  & 307  & 2  & 54     & 2 \\
        127  & 2 & 27     & 2  & 433  & 2  & 27     & 2  \\
        307  & 2 & 54     & 2  & 2017 & 2  & 9      & 2  \\
        397  & 2 & 54     & 2  & 2341 & 2  & 108    & 2  \\
        433  & 2 & 27     & 2  & 3331 & 2  & 18     & 2  \\
        739  & 2 & 36     & 2  & 3457 & 2  & 27     & 2  \\
        \hline
        19   & 0 & 3      & 1  & 19   & 0 & 3       & 1   \\
        37   & 0 & 3      & 1  & 37   & 0 & 3       & 1   \\
        \hline
    \end{tabular}
    
    \begin{tabular}{|c|c|c|c||c|c|c|c|}
        \hline
        \multicolumn{4}{|c||}{Theorem \hyperlink{thm:A}{A},  $n= 3\ell$, $\ell \equiv 7\pmod{9}$} & \multicolumn{4}{|c|}{Theorem \hyperlink{thm:A}{A},  $n= 3\ell^{2}$, $\ell \equiv 4\pmod{9}$} \\
        \hline
        $\ell$ & $r_{\operatorname{al}}(3\ell)$ & $h(12\ell)$   & $h_3{(12\ell)}$ & $\ell$ & $r_{\operatorname{al}}(3\ell^2)$ & $h(18\ell)$  & $h_{3}(18 \ell)$ \\
        \hline
        61   &2  & 18    & 2    & 193  &2 & 18    & 2    \\
        151  &2  & 108   & 2    & 499  &2 & 108   & 2    \\
        367  &2  & 18    & 2    & 1759 &2 & 18    & 2    \\
        439  &2  & 72    & 2    & 2389 &2 & 360   & 2    \\
        619  &2  & 90    & 2    & 2713 &2 & 72    & 2    \\
        727  &2  & 54    & 2    & 3217 &2 & 54    & 2    \\
        \hline
        43   &0  & 12    & 1    & 13   &0  & 6    & 1    \\
        79   &0  & 3    & 1     & 229  &0  & 3    & 1    \\
        \hline
    \end{tabular}
\end{table}

\end{document}